\documentclass[11pt,reqno]{amsart}
\usepackage[a4paper, hmargin={2.7cm,2.7cm},vmargin={3.3cm,3.3cm}]{geometry}
\usepackage{latexsym,amsmath,amssymb,amscd}
\usepackage[noadjust]{cite}
\usepackage{amsfonts}
\usepackage{amsthm}
\usepackage{enumerate}
\usepackage{amsthm}
\usepackage{todonotes}

\usepackage[all]{xy}

\usepackage{etoolbox}

\numberwithin{equation}{section}

\makeatletter
\patchcmd{\ttlh@hang}{\parindent\z@}{\parindent\z@\leavevmode}{}{}
\patchcmd{\ttlh@hang}{\noindent}{}{}{}
\makeatother

\DeclareMathOperator*{\tr}{tr}

\newtheorem{theorem}{Theorem}[section]
\newtheorem{lemma}[theorem]{Lemma}
\newtheorem{proposition}[theorem]{Proposition}
\newtheorem{corollary}[theorem]{Corollary}

\theoremstyle{definition}

\newtheorem{examplex}[theorem]{Example}
\newenvironment{example}
  {\pushQED{\qed}\examplex}
  {\popQED\endexamplex}

\newtheorem{remark}[theorem]{Remark}

\newcommand{\Ad}{{\rm Ad}}

\newcommand{\ad}{{\rm ad}}

\newcommand{\de}{{\rm d}}
\newcommand{\ee}{{\rm e}}
\newcommand{\End}{{\rm End}\,}
\newcommand{\conv}{{\rm conv}}

\newcommand{\ie}{{\rm i}}
\newcommand{\im}{{\rm Im}\,}

\newcommand{\ind}{{\rm ind}}
\newcommand{\Ker}{{\rm Ker}\,}

\newcommand{\pker}{{\rm pker}\,}
\newcommand{\Prim}{{\rm Prim}\,}
\newcommand{\Ran}{{\rm Ran}\,}

\renewcommand{\Re}{{\rm Re}\,}

\newcommand{\Tr}{{\rm Tr}\,}
\newcommand{\CC}{{\mathbb C}}

\newcommand{\RR}{{\mathbb R}}

\newcommand{\ZZ}{{\mathbb Z}}

\newcommand{\Hc}{\Hpi}

\newcommand{\Oc}{{\mathcal O}}

\newcommand{\Uc}{{\mathcal U}}

\renewcommand{\gg}{{\mathfrak g}}
\newcommand{\hg}{{\mathfrak h}}

\newcommand{\pg}{{\mathfrak p}}

\newcommand{\dual}[2]{\langle #1, #2\rangle}
\newcommand{\scalar}[2]{( #1, #2)}

\newcommand{\Hpi}{\mathcal{H}_{\pi}}

\newcommand{\vol}{\mathrm{vol}}

\def\XXint#1#2#3{{\setbox0=\hbox{$#1{#2#3}{\int}$ }
\vcenter{\hbox{$#2#3$ }}\kern-.6\wd0}}

\title[]{Symplectic projective orbits \\ of unimodular exponential Lie groups}
\author{Ingrid Belti\c t\u a}
\address{Institute of Mathematics "Simion Stoilow" of the Romanian Academy, PO Box 1-764, 014700 Bucharest, Romania}
\email{ingrid.beltita@gmail.com, Ingrid.Beltita@imar.ro}

\author{Jordy Timo van Velthoven}
\address{Faculty of Mathematics,
University of Vienna, 
Oskar-Morgenstern-Platz 1,
1090 Vienna, Austria}
\email{jordy-timo.van-velthoven@univie.ac.at}

\makeatletter
\@namedef{subjclassname@2020}{\textup{2020} Mathematics Subject Classification}
\makeatother



\begin{document}

\begin{abstract}
For an exponential Lie group $G$ and an irreducible unitary representation $(\pi, \Hpi)$ of $G$, we consider the natural action defined by $\pi$ on the projective space of $\Hc$, and show 
that the stabilisers of this action coincide with the projective kernel of $\pi$. 
Using this, we prove that,  if $G/\pker(\pi)$ is unimodular, then  $\pi$ admits a symplectic projective orbit if and only if $\pi$ is square-integrable modulo its projective kernel  $\pker(\pi)$.
\end{abstract}

\maketitle

\section{Introduction}
Let $(\pi, \Hpi)$ be an irreducible unitary representation of a connected Lie group $G$. 
Then $G$ defines an action on the associated projective Hilbert space $P(\Hpi)$ by $g \cdot [\eta] = [\pi(g) \eta]$, where $[\eta] := \mathbb{C} \eta$ denotes the ray generated by $\eta \in \Hpi \setminus \{0\}$. An orbit of this action,
\begin{align} \label{eq:CSorbit}
G \cdot [\eta] = \{ [\pi(g) \eta] : g \in G \},
\end{align}
is often referred to as a (Perelomov-type) \emph{coherent state orbit}  \cite{perelomov1972coherent, debievre1989coherent, lisiecki1990kaehler}. 

A classical problem in the study of coherent state orbits is to determine the representations $\pi$ and vectors $\eta$ for which the associated  orbit \eqref{eq:CSorbit} admits an additional structure, such as a symplectic stucture \cite{kostant1982symplectic, lisiecki1993symplectic, debievre1989coherent} or even a K\"ahler structure \cite{lisiecki1991classification, neeb1996coherent, lisiecki1990kaehler}. 
One motivation for studying the existence of symplectic and complex coherent state orbits is their use in geometric quantisation \cite{odzijewicz1992coherent, guillemin1984symplectic} and Berezin quantisation \cite{berezin1975general, rawnsley1990quantization}, respectively. In addition, symplectic and complex coherent state orbits are naturally related to convexity properties of moment maps of unitary representations, see, e.g., \cite{arnal1992convexity, neeb1996coherent}.
Although groups and representations admitting complex coherent state orbits are quite well-understood \cite{lisiecki1991classification}, the picture is far less complete for symplectic coherent state orbit, see, e.g., \cite[Conj. 11.1]{lisiecki1993symplectic}.

The aim of the present paper is to study the existence of symplectic coherent state orbits for exponential Lie groups, that is, Lie groups whose exponential map is a diffeomorphism; any such group is solvable. 
Representations admitting such symplectic orbits are generally referred to as \emph{coherent state representations} in the literature \cite{lisiecki1991classification, neeb1996coherent, moscovici1978coherent}, and are often assumed to have a discrete kernel.
As a consequence of more general results, we obtain the following characterisation of coherent state representations of unimodular exponential Lie groups.

\begin{theorem} \label{thm:intro}
    Let $G$ be a unimodular exponential Lie group and let $(\pi, \Hpi)$ be an irreducible representation of $G$ with discrete kernel. The following assertions are equivalent:
    \begin{enumerate}[{\rm(i)}]
        \item There exists a smooth vector $\eta \in \Hpi^{\infty} \setminus \{0\}$ such that $G \cdot [\eta]$ is symplectic.
        \item For every smooth vector $\eta \in \Hpi^{\infty} \setminus \{0\}$, the orbit $G \cdot [\eta]$ is symplectic.
        \item $\pi$ is square-integrable modulo the centre.
    \end{enumerate}
    If one of the above conditions holds true, then $G$ has nontrivial centre. 
\end{theorem}

We provide an example (cf. Example \ref{ex:counterexample_nonunimodular}) of a nonunimodular group that admits symplectic orbits also for non-square-integrable representations, showing that Theorem \ref{thm:intro} might fail for nonunimodular groups.

For nilpotent Lie groups, Theorem~\ref{thm:intro} was announced and stated as \cite[Thm. 4.1]{lisiecki1993symplectic}; however, to the best of our knowledge, a proof  has not been outlined or published. On the other hand, the class of unimodular exponential Lie groups is larger than the class of nilpotent Lie groups. 
Indeed, for an exponential Lie group $G$ with Lie algebra $\mathfrak{g}$, the unimodular function is given by 
$ \Delta_G(\exp X) = e^{\Tr (\ad_X)}$, $X\in \gg$,
Hence the condition that $G$ is unimodular is equivalent with $\Tr(\ad_X)=0$ for every $X\in \gg$.  
This is the case, for instance, for semidirect products of 
the form $ V \rtimes_{\alpha_D} \RR$, where $V$ is a finite dimensional real vector space, $D\colon V \to  V$ is a linear map with $\Tr(D)=0$ with any purely imaginary eigenvalue, and the action $\alpha_D\colon \RR \to \End(V)$ is given by $\alpha_D(t)=e^{tD}$, $t\in \RR$. Another relevant example is given in Example \ref{ex:unimodular}.

Beyond unimodular groups, the relation between coherent state representations and square-integrable representations is more delicate. Although any exponential Lie group admitting a coherent state representation also admits a square-integrable representation with the same projective kernel, and, conversely, any square-integrable representation is a coherent state representation (see Proposition \ref{prop:non1mod} for both facts), there might exist symplectic orbits also for non-square-integrable representations, cf. Example \ref{ex:counterexample_nonunimodular}.

Lastly, we give an application of our results to Perelomov's completeness problem \cite{perelomov1972coherent}, and show that (in the case of exponential Lie groups) necessary conditions for the completeness of coherent state subsystems can be obtained from criteria for the cyclicity of restrictions of associated projective representations obtained in \cite{bekka2004square, enstad2022density, romero2022density}. 

The paper is organised as follows. Section \ref{sec:irreducible} contains preliminary results on the projective kernel of an irreducible representation and square-integrable representations. Section \ref{sec:symplectic} is devoted to the study of the existence of symplectic coherent state orbits and the square-integrability of the representation, including, among others, a proof of Theorem \ref{thm:intro}. Lastly, in Section \ref{sec:perelomov}, we present an application to Perelomov's completeness problem for coherent state subsystems in the setting of exponential Lie groups.  

\subsection*{Notation}
Lie groups will be denoted with capital letter $G$, $H$, etc, while their respective Lie algebras
are denoted with the corresponding gothic letters $\gg$, $\hg$, etc.
For an irreducible representation $\pi\colon G \to \Uc(\Hpi)$ we denote by the same letter $\pi$ its extension 
to the group $C^*$-algebra $C^*(G)$, and its unitary equivalence class in $\widehat G =\widehat{C^*(G)}$.

For  a complex vector space $\mathcal{H}$, we denote by  $P(\mathcal{H})$ its projective
space, that is, the set of all one-dimensional subspaces of $\mathcal{H}$. 
It can be alternatively described as the set of equivalence classes
for the equivalence relation on $\mathcal{H} \setminus \{0\}$ defined by $\eta_1 \sim \eta_2$ if there is $\lambda \in \CC$
such that $\eta_1=\lambda \eta_2$. 
We denote by $[\eta] := \mathbb{C} \eta$ the equivalence class of $\eta \in \mathcal{H} \setminus \{0\}$. The unit circle is denoted by $\mathbb{T} := \{ z \in \mathbb{C} : |z|=1\}$.

\section{Irreducible representations of exponential groups} \label{sec:irreducible}
Let $G = \exp(\mathfrak{g})$ be an exponential Lie group and let $(\pi, \Hpi)$ be an irreducible unitary representation of $G$. 
By the coadjoint orbit method,
there exists $\ell \in \mathfrak{g}^*$ and a real polarisation $\mathfrak{h}$ of $\ell$ such that $\pi$ is unitarily equivalent to the monomial representation $\pi_{\ell} := \ind_H^G \chi_{\ell}$, where $H = \exp(\mathfrak{h})$ and
\[
 \chi_{\ell} (\exp X) = e^{i \dual{\ell}{X}}, \quad X \in \mathfrak{h}.
\]
The coadjoint orbit $\mathcal{O}_{\ell} := \Ad^*(G) \ell$ associated with the equivalence class $[\pi] \in \widehat{G}$ is often simply denoted by $\mathcal{O}_{\pi}$. 
See  \cite{arnal2020representations, fujiwara2015harmonic} for background on the coadjoint orbit method.

\subsection{Projective kernel and stabiliser}

The projective kernel of $\pi$ is the closed normal subgroup
\[
 \pker(\pi) = \{ x \in G : \pi(x) \in \mathbb{T} \cdot I_{\Hpi} \}.
\]
By \cite[Thm. 2.1]{bekka1990complemented}, the group $\pker(\pi)$ is equal to the intersection $\bigcap_{\ell \in \mathcal{O}_{\pi}} G(\ell)$ of the stabiliser subgroups $G(\ell) = \{ x \in G : \Ad^* (x) \ell = \ell \}$. In particular, this implies that $\pker(\pi)$ is connected. 
We denote  by $\ker(\pi) = \{x \in G : \pi(x) = I_{\Hpi} \}$ the kernel of $\pi$, which is a closed normal subgroup of $G$. 

The following remark will be used repeatedly.

\begin{remark}\label{centre}
\normalfont
If $\pi\colon  G \to \mathcal{U}(\Hpi)$ is an irreducible unitary representation of an exponential Lie group, then the centre of $G/\ker(\pi)$ is compact and equal to $\pker(\pi)/\ker(\pi)$.

Indeed, an element $x \ker(\pi)$ is in the centre $Z(G/\ker(\pi))$ if and only if $xyx^{-1}y^{-1} \in \ker(\pi)$ for every $y\in G$.
This is in turn equivalent with $\pi(x) \pi(y)=\pi(y)\pi(x)$ for every $y\in G$. 
But since 
$\pi$ is irreducible, this is  equivalent with $\pi(x)\in \CC \cdot I_{\Hc}$, that is, $x\in \pker(\pi)$.
The fact that  $\pker(\pi)/\ker(\pi)$ is compact follows from \cite[Thm.~2.1]{bekka1990complemented}.
\end{remark}

The group $G$ acts continuously on the projective space 
$P(\Hc)$ (equipped with the quotient topology)
by $(x, [\eta])\mapsto x\cdot [\eta] = [\pi(x)\eta]$. 
The projective stabiliser of a fixed  $\eta \in \Hpi \setminus \{0\}$, that is, the stabiliser with respect to the above action, is then
\[ G_{[\eta]} = \{ x \in G : \pi(x) \eta \in \mathbb{T} \cdot \eta \}. \]
Generally, $\pker(\pi) \subseteq G_{[\eta]}$. For exponential Lie groups, the reverse inclusion also holds.

\begin{proposition} \label{prop:stabilizer_kernel}
Let $(\pi, \Hpi)$ be an irreducible representation of an exponential Lie group. Then $G_{[\eta]} = \pker(\pi)$ for every $\eta \in \Hpi \setminus \{0\}$.
\end{proposition}
\begin{proof}
It is enough to prove the proposition for  $\eta\in \Hc \setminus \{0\}$ with $\| \eta \| =1$. 
Arguing by contradiction, assume that $\pker(\pi)\subsetneq G_{[\eta]}$, so that $G_{[\eta]}/\pker(\pi)$ 
 is nontrivial. 
Consider the matrix coefficient  $ G\to \CC$, 
$x\mapsto \scalar{\eta}{\pi(x)\eta}$. 
It gives a well-defined, continuous function $\dot{x} \mapsto 
f^\pi_\eta(\dot{x})=\vert \scalar{\eta}{\pi(x)\eta} \vert $ on $G/\pker(\pi)$. 
By \cite[Thm.~7.1 and Prop.~4.1]{howe1979asymptotic}, the function
$f^\pi_\eta$ vanishes at infinity. 
On the other hand, $f^\pi_\eta(\dot{x})=1$ for $x \in G_{[\eta]}/\pker(\pi)$.  
It follows that the closed group $G_{[\eta]}/\pker(\pi)$
is a compact subset, hence a compact subgroup,  of  
  $G/\pker(\pi)$. 
  Since $G$ is exponential and $\pker(\pi)$ is connected,
  the quotient $G/\pker(\pi)$ is an exponential
Lie group (cf. \cite[Cor. 1.8.5]{arnal2020representations}).
Hence, the compact subgroup $G_{[\eta]}/\pker(\pi)$  of $G/\pker(\pi)$  
must be trivial,
which is the required contradiction.
\end{proof}

The following example demonstrates that Proposition \ref{prop:stabilizer_kernel} fails for possibly nonexponential solvable Lie groups.

\begin{example} \label{ex:stabilizer}
Let $G = \mathbb{C} \rtimes \mathbb{R}$ be the connected simply connected solvable Lie group with multiplication
\[
 (z, t) (z', t') = (z + e^{it} z', t + t').
\]
This is the universal cover of the group $E_2$ of Euclidean displacements.
The group $G$ is not exponential, see, e.g., \cite[Ex. 1.7.14]{arnal2020representations}.
The basis $\{X,Y,T\}$ of the Lie algebra $\mathfrak{g}$ of $G$  satisfies
\[
 [T,X] = Y, \quad [T, Y] = -X, \quad [X,Y] = 0.
\]

For each $r>0$, a unitary representation $\pi_r$ of $G$ acting on ${\mathcal H}:=L^2 (\mathbb{R} / 2\pi \mathbb{Z})$ is given by 
\[
 \pi_{r} (z, t) \eta( s + 2 \pi \mathbb{Z}) = e^{i \Re (z e^{-i s} r)} \eta(s - t + 2\pi \mathbb{Z}),
\]
cf. \cite[Ex. 3.3.28]{arnal2020representations}. 
Each such $\pi_r$ is irreducible, cf. \cite[Ex. 3.4.10]{arnal2020representations}.

For $k \in \mathbb{Z}$, consider the nonzero vector  $\eta_k (s) = e^{i k s}$.
Then 
\begin{align*}
 (\pi_r (z,t) \eta_k) (s) = e^{- i k t} e^{i \Re (z e^{-i s} r)} \eta_k(s).
\end{align*}
This shows that $H:= \{0\} \times \mathbb{R}=G_{[\eta_k]}$ for any $k \in \mathbb{Z}$.
On the other hand, note that $(\eta_k + \eta_{-k}) (s)  = 2 \cos (ks)$, and hence 
\[
 \big( \pi_{r} (0,t) (\eta_k + \eta_{-k}) \big) (s) = 2 \cos (k(s-t))
\]
for $s,t \in \mathbb{R}$.
Let $t \in \mathbb{R}$ be such that there exists $\lambda \in \mathbb{T}$ satisfying
\[
 (\pi_r (0,t) (\eta_k+\eta_{-k}))(s) = \lambda (\eta_k+\eta_{-k})(s) \quad \text{for any} \quad s \in \mathbb{R},
\]
that is, 
\[
2 \cos (k(s-t))=  \lambda  2 \cos (ks) \quad \text{for any} \quad s \in \mathbb{R}.
\]
Then necessarily  $\lambda = \pm 1$ and $\cos(k(s-t)) = \pm \cos(ks)$ for any $s \in \mathbb{R}$, which implies that $t \in \frac{\pi}{k}
\mathbb{Z}$. 
It follows therefore that \[
G_{[\eta_k + \eta_{-k}]} \cap H = \{0 \} \times  \frac{\pi}{k}
\mathbb{Z} \subsetneq G_{[f_k]} \cap H, \]
and hence $G_{[\eta_k + \eta_{-k}]}\subsetneq G_{[\eta_k]}$.
That is, $G_{[\eta]}$ is not constant for all $\eta \in {\mathcal H} \setminus \{0\}$. 

Let us compute $\pker(\pi_r)$. For this, let  $\xi \in {\mathcal H}$ be given by $\xi (s) = \cos (s) +\cos (2s)$. 
Then 
$$ (\pi_r(0, t)\xi)(s) = \cos (s-t)+ \cos (2(s-t)) \quad \text{for} \quad t \in \mathbb{R}$$
If we assume that $(0, t) \in G_{[\xi]}$, then there is $\lambda\in \mathbb{T}$ such that 
$$\cos (s-t)+ \cos (2(s-t)) = \lambda \big( \cos (s) + \cos (2s) \big), \qquad s\in [0, 2\pi).$$
Taking $s=0$, and $s=\pi$, we get, respectively, 
\begin{align*}
\cos (t) +  \cos (2t) & = 2\lambda\\
-\cos (t) +  \cos (2t) & = 0.
\end{align*}
This shows that $\cos (t) =\cos (2t) = \lambda$, so that $\lambda =1$ and 
$t \in 2\pi \ZZ$, and hence $G_{[\xi]}\cap H =G_{[\xi]}\cap G_{[\eta_k]}=
\{0\} \times 2\pi \ZZ$.
Since  $\pker{(\pi_r)} \subseteq \bigcap_{\eta \in \Hc\setminus \{0\}} G_{[\eta]}$, it follows that
$ \pker{(\pi_r)} \subseteq  \{0\} \times 2\pi\mathbb{Z}.$
On the other hand, from the definition of $\pi_r$, we see that 
 $\{0\} \times 2\pi\mathbb{Z}\subseteq \ker{(\pi_r)}\subseteq \pker{(\pi_r)} $, and thus
 $$ \pker{(\pi_r)} =\{0\} \times 2\pi\mathbb{Z} =\ker{(\pi_r)}.$$
 Note that in this case  the subgroup $\pker{(\pi_r)}$ is not even connected. 
 \end{example}

\subsection{Square-integrable representations}
For an irreducible representation $(\pi, \Hpi)$ of a connected Lie group $G$, let $H$ be a closed normal subgroup of $G$ that is contained in $\pker(\pi)$.
The representation $\pi$ of $G$ is said to be \emph{square-integrable modulo $H$} if there exists $\eta \in \Hpi \setminus \{0\}$ such that the well-defined continuous function $\dot{x} \mapsto |\scalar{ \eta}{\pi(x) \eta }|$ is square-integrable on $G/H$, with respect to the left Haar measure on $G/H$.
 
In the following proposition we gather the  different characterisations of the irreducible representations of exponential Lie groups that are square-integrable modulo their projective kernel.

\begin{proposition} \label{prop:sq}
Let $\pi$ be an irreducible representation of an exponential Lie group $G$. 
The following assertions are equivalent:
\begin{enumerate}[{\rm(i)}]
\item\label{sq-prop_item1} $\pi$ is square-integrable modulo $\pker(\pi)$.
\item\label{sq-prop_item2} The coadjoint orbit $\Oc_\pi$ is open in its affine hull.
\item\label{sq-prop_item3} $\gg(\ell)$ is an ideal in $\gg$ for some
$\ell \in \Oc_\pi$.
\item\label{sq-prop_item4} $G(\ell)$ is a normal subgroup of $G$ for some  
$\ell \in \Oc_\pi$. 
\end{enumerate}
If one of the above \eqref{sq-prop_item1}--\eqref{sq-prop_item4}
holds, then $G(\ell)$ and $\gg(\ell)$ are independent of $\ell \in \Oc_\pi$.
\end{proposition}

\begin{proof}
Let $\ker(\pi) \subseteq G$ be the kernel of the representation $\pi \colon G \to \Uc(\Hc)$. 
Then,  by \cite[Cor.~2.1]{bekka1990complemented},  $\pi$ is square-integrable modulo $\pker(\pi)$ if and only if 
it is square-integrable modulo $\ker(\pi)$.
Thus \eqref{sq-prop_item1} $\iff$ \eqref{sq-prop_item2} 
 follows directly from  \cite[Thm.~2.4.1]{moscovici1978coherent} and Remark~\ref{centre}.
The equivalence \eqref{sq-prop_item2} $\iff$ \eqref{sq-prop_item3} follows from \cite[Lem.~2.1.2]{moscovici1978coherent}, or 
\cite[Prop.~3.1]{baklouti2023open}, 
while \eqref{sq-prop_item3} $\iff$ \eqref{sq-prop_item4} is a consequence of the fact that $G$ is exponential, hence $G(\ell)=\exp (\gg(\ell))$. 
The equality $\gg(\Ad^*(\exp X)\ell)= \Ad (\exp(-X)) \gg(\ell)$ implies the last assertion of the statement. 
\end{proof}

By \cite[Lemma~5.3.7]{fujiwara2015harmonic}, the projective kernel $\pker(\pi)$ of $\pi$ is the largest closed
normal subgroup of $G(\ell)$. 
This immediately yields the following consequence.

\begin{corollary}\label{sq-cor}
The representation $\pi$ is square-integrable modulo $\pker(\pi)$ if and only if $\pker(\pi) = G(\ell)$  for some (and then all) $\ell\in \Oc_\pi$.
\end{corollary}

The following lemma provides a sufficient condition on a representation and a group under which the associated coadjoint orbit is closed and affine. 
This result plays an essential role in proving our main theorem. 

\begin{lemma}\label{rose}
Let $G$ be a exponential Lie group with its Lie algebra $\gg$, and let $\pi\colon G \to \Uc(\Hc)$ be
an irreducible unitary representation of $G$ such that $G/\pker(\pi)$ is unimodular.
Assume that $\pi$ is square-integrable modulo $\pker(\pi)$.
Then the corresponding coadjoint orbit $\Oc_\pi$ is closed and affine, namely 
$$ \Oc_\pi=\ell + \gg(\ell)^\perp$$
for every $\ell \in \Oc_\pi$.
\end{lemma}

\begin{proof}
Let $\ker(\pi) \subseteq G$ be the kernel of the representation $\pi\colon G \to \Uc(\Hc)$, and denote by
$G':=G/\ker(\pi)$ the associated quotient group. 
Then $G'$ is a solvable Lie group, connected since the map $p\colon G\to G'$ is continuous and $G$ is connected. 
The centre of $G'$ is $K= \pker(\pi)/\ker(\pi)$, which is compact (cf. Remark \ref{centre}), and hence unimodular. 
On the other hand, $G'/K = G/\pker(\pi)$ is also unimodular, hence so is $G'$ itself, see, e.g., \cite[Rem.~6]{aniello2006square}.

Since $\pi$ is square-integrable modulo $\ker(\pi)$ by \cite[Cor. 1]{bekka1990complemented}, it follows that the irreducible representation $\pi'=\pi\circ p $ is square-integrable (in the strict sense) on $G'$.
Since $G'$ is unimodular, by \cite[Cor. 3.9]{rosenberg1978square}, the kernel $\Ker(\pi')$ of the representation $\pi'$ considered on 
$C^*(G')$
is a closed point in the primitive spectrum $\Prim C^*(G')$ of $C^* (G')$, and hence it is a closed point in
$\widehat{G'}$, since $G'$ is of type I.
The map $\hat{p}\colon  \widehat{G'}\to \widehat{G}$, $\hat{p}(\rho)=  \rho\circ p$  is a 
homeomorphism between $\widehat{G'} $ and a closed subspace of $\widehat{G}$, see, e.g., 
\cite[Prop. 1.C.11 (3) \& Rem. 1.C.12 (2)]{bekka2020delaharpe}. 
Then the singleton $\{\pi\} =\hat {p}^{-1}(\{\pi'\})$ is closed in $\widehat{G}$, and thus $\pi$ is a CCR representation. 
Hence, by \cite[Thm.~1]{pukanszky68}, the coadjoint orbit $\Oc_\pi$ is  closed. 

On the other hand, $\Oc_\pi$ is an open subset of the affine space $\ell +\gg(\ell)^\perp$ for every $\ell \in \Oc_\pi$, since $\pi$ is square-integrable modulo $\pker \pi$, cf. Proposition \ref{prop:sq}. Therefore, $\Oc_\pi=\ell +\gg(\ell)^\perp$, as desired. 
\end{proof}

The following example provides an example of a unimodular group $G$ for which Lemma~\ref{rose} might fail without the unimodularity assumption on $G/\pker(\pi)$.

\begin{example} \label{ex:unimodular}
    Let $G$ be the connected, simply connected completely solvable Lie group with Lie algebra $\gg = \mathrm{span} \{X_1, X_2, X_3, X_4, X_5\}$ satisfying the nontrivial bracket relations
$$
[X_2, X_3] =X_1,\;  [X_2, X_5]=X_2, \;  [X_3, X_5] = -X_3,\;  [X_4, X_5] = X_1.
 $$
 Denote by $X_j^*$, $j=1, \dots, 5$, the dual basis in $\gg^*$.
For $\ell= X_3^*$, a direct calculation gives  $\gg(\ell)=\text{span}\{X_1, X_2, X_4\}$, which is an ideal in $\mathfrak{g}$. 
Therefore, the associated irreducible representation $\pi_\ell$ is square-integrable modulo its projective kernel $\pker(\pi_{\ell})$ by Proposition \ref{prop:sq}. 
In addition, Corollary \ref{sq-cor} yields $\pker(\pi_{\ell}) = G(\ell)$, which is the connected Lie subgroup of $G$ with Lie algebra $\gg(\ell)$. 
The quotient Lie algebra $\gg / \gg(\ell)$ is isomorphic to the Lie algebra of the affine group, and hence the quotient group $G/\pker(\pi_{\ell}) = G/G(\ell)$ is nonunimodular.

For showing that $\Oc_\ell \subseteq \ell+\gg(\ell)^\perp$, assume towards a contradiction that $\Oc_\ell = \ell+\gg(\ell)^\perp$.
Since $-\ell\vert_{\gg(\ell)} =0$, it follows that $-\ell \in \gg(\ell)^\perp$. 
If our assumption were true, then $0\in \Oc_\ell$, that is, $\Oc_\ell=\{0\}$.
This is a contradiction, hence 
$\Oc_\ell \subseteq \ell +\gg(\ell)$.
\end{example}

\section{Symplectic projective orbits and square-integrable representations} \label{sec:symplectic}

This section is devoted to the relation between the existence of symplectic coherent state orbits and the square-integrability of a representation. 

\subsection{Symplectic coherent state orbits}
Let $G$ be a connected Lie group with Lie algebra $\mathfrak{g}$  and let $(\pi, \Hpi)$ 
be an irreducible unitary representation of $G$. 
Denote by $\Hpi^{\infty}$  the space of smooth vectors of $\pi$, i.e., the family of vectors $\eta \in \Hpi$ such that the orbit map $x \mapsto \pi(x) \eta$ is smooth. 

The projective space $P(\Hpi)$ is then a
Hilbert
manifold with respect to the local charts $(U_\eta, \varphi_\eta)$ at a point $[\eta] \in P(\Hpi)$, 
$\eta \in \Hpi \setminus \{0\}$, defined by
$$ U_\eta=\{ [\xi]\in P(\Hc) \colon  \scalar{\xi}{\eta} \neq 0\}, \quad \varphi_\eta([\xi])= \Vert \eta\Vert^2 \frac{\xi}{\scalar{\xi}{\eta}}-\eta.$$
The $2$-form on 
 $P(\Hc)$
 defined in the chart $\varphi_\eta$ by 
$$ \omega^{P(\Hc)}_{[\eta]}(w_1, w_2)= 2\,
\frac{\im \scalar{T_{[\eta]} (\varphi_\eta)( w_1)}{T_{[\eta]}( \varphi_{\eta})(w_2)}}{\Vert\eta\Vert^2},
\quad 
w_1, w_2\in T_{[\eta]}(P(\Hpi)), 
$$
makes 
 $P(\Hc)$
  into a symplectic Hilbert manifold. 
  Here, for a smooth map $f\colon M\to N$ between Fr\'echet manifolds $M$, $N$, and $p\in M$, $T_p(f) \colon T_p M \to T_{f(p)} N$ denotes the linear tangent map. 
  
  Similarly, the projective space $P(\Hc^\infty)$ is a
Fr\'echet 
manifold with respect to the local charts $(U_\eta, \phi_\eta)$ at a point $[\eta] \in P(\Hc^\infty)$, 
$\eta \in \Hc^\infty \setminus \{0\}$, defined by 
$$ U_\eta=\{ [\xi]\in P(\Hc^\infty) \mid \scalar{\xi}{\eta} \neq 0\}, \quad \phi_\eta([\xi])= \Vert \eta\Vert^2 \frac{\xi}{\scalar{\xi}{\eta}}-\eta.$$
The natural inclusion $i\colon P(\Hc^\infty) \to P(\Hc)$ is smooth, and 
the $2$-form $\omega^{P(\Hc^\infty)}= i^* \omega^{P(\Hc)}$ of 
 $P(\Hc^\infty)$, given by
\begin{equation}\label{omega}
\omega^{P(\Hc^\infty)}_{[\eta]}(w_1, w_2)= 2\,
\frac{\im \scalar{T_{[\eta]}(\phi_{\eta})(w_1)}{T_{[\eta]}(\phi_{\eta})(w_2)}}{\Vert\eta\Vert^2}, \quad 
w_1, w_2\in T_{[\eta]}(P(\Hpi^\infty)), 
\end{equation}
is symplectic. 
The mapping $i$ is an immersion. 
(See, e.g., \cite[Sect.~4.3]{arnal2011universal} for more
details.)

As mentioned before, the group $G$ acts on the projective spaces 
$P(\Hc)$ and $P(\Hc^\infty)$
by $(x, [\eta])\mapsto x\cdot [\eta] = [\pi(x)\eta]$.
The action on $P(\Hc^\infty)$ is smooth, Hamiltonian and its moment map 
$J_\pi \colon  P(\Hpi^\infty)\to \gg^*$
is given by
\begin{equation}\label{m-map}
\dual{J_\pi([\eta])}{X}=\frac{1}{i} \frac{\scalar{\de \pi(X)\eta}{\eta}}{\Vert\eta\Vert^2}, \quad [\eta]\in P(\Hpi^\infty), \; X\in \gg,
\end{equation}
see, e.g., \cite[Prop.~4.6]{arnal2011universal}.

For $\eta \in \Hc\setminus \{0\}$, the orbit 
$$G\cdot [\eta] = \{[\pi(x)\eta] \mid x\in G\}$$ 
is a smooth immersed submanifold
of $P(\Hc)$ if (and only if) $\eta\in \Hc^\infty$
(see \cite[Prop.~2.4]{lisiecki1990kaehler}).
If $\eta \in \Hc^\infty\setminus \{0\}$,  $G\cdot [\eta] \subseteq P(\Hc^\infty)$, and since the inclusion
$i\colon P(\Hc^\infty)\hookrightarrow P(\Hc)$ is an immersion, it follows that
$G\cdot [\eta]$ is a smooth immersed submanifold
of $P(\Hc^\infty)$.
In that case, let $\iota\colon G \cdot [\eta] \to P(\Hc^\infty)$ be the canonical immersion. 
The pullback $\iota^* \omega^{P(\Hc^\infty)}=:\omega^{G\cdot [\eta]}$ is a $G$-invariant
closed $2$-form on $G\cdot [\eta]$. 
The orbit $G\cdot [\eta]$ is \emph{symplectic} if the form $\omega^{G\cdot [\eta]}$ is symplectic. 

The following lemma  is well-known for finite-dimensional representations of connected Lie groups, see, e.g., \cite[Theorem 26.8]{guillemin1984symplectic}. As we are not aware of a reference in the infinite-dimension case, we provide its short proof.
\begin{lemma} \label{lem:symplectic}
 Let $(\pi, \Hpi)$ be an irreducible representation of a connected Lie group $G$. 
 For $\eta\in \Hc^\infty\setminus \{0\}$,
 the orbit $G \cdot [\eta]$ is symplectic if and only if $G_{[\eta]}$ is an open subgroup of $G(J_{\pi} ([\eta]))$.
\end{lemma}
\begin{proof} 
First note that  $G_{[\eta]} \subseteq G(J_\pi([\eta]))$ for $\eta \in \Hpi^{\infty} \setminus \{0\}$, hence the claim that
$G_{[\eta]}$ is an open subgroup of $G(J_{\pi} ([\eta]))$ is equivalent with 
$ \gg_{[\eta]}= \gg (J_\pi([\eta]))$.

Fix $\eta\in \Hc^\infty\setminus \{0\}$ and denote $\Omega:= G\cdot [\eta]$.
The form $\omega^\Omega$ is $G$-invariant, hence it is enough to show that  
it
is nondegenerate at $[\eta]$ if and only if $ \gg_{[\eta]}= \gg (J_\pi([\eta]))$.

Let $q\colon G \to \Omega \simeq G/G_{[\eta]}$ be the orbit map, and
denote $\alpha \colon G \to P(\Hpi^\infty)$,   $\alpha(g)= [\pi(g) \eta]$.
Then $\alpha = \iota \circ q$, hence 
$$ \alpha^* (\omega^{P(\Hpi^\infty)}) = q^* (\iota^* \omega^{P(\Hpi^\infty)})= q^* (\omega^{\Omega}).$$
Denote $\theta:=q^* (\omega^{\Omega})$. 
Then $\omega^\Omega_{[\eta]} $ is nondegenerate if and only if 
for $\theta _1 \colon  T_1 G \times T_1G =\gg \times \gg \to \RR$ we have 
$ \gg^{\perp_{\theta_1}} := \{ X \in \mathfrak{g} : \theta_1 (X,Y) = 0 \; \text{for all $Y  \in \mathfrak{g}$} \} = \gg_{[\eta]}$. 

On the other hand, $\theta = \alpha^* (\omega^{P(\Hpi^\infty)})$, hence for $X, Y\in \gg$, 
$$\theta_1 (X, Y) = \omega^{P(\Hpi^\infty)}_{[\eta]}( T_1(\alpha)(X) , T_1(\alpha)(Y)).$$
For $X\in T_1G =\gg$, we have 
\begin{align*}
T_1 (\phi_\eta \circ \alpha) (X) & =\frac{d}{dt}\Big\vert_{t=0} (\phi_\eta \circ \alpha)(\exp (tX)) \\
& = \de\pi(X) \eta-\frac{\scalar{\de \pi(X) \eta}{\eta}}{\Vert \eta \Vert^2 }\eta.
\end{align*}
Therefore, by \eqref{omega}, we can write 
\begin{align*}\theta_1 (X, Y) &=
 \omega^{P(\Hpi^\infty)}_{[\eta]}( T_1(\alpha)(X) , T_1(\alpha)(Y))\\
&=\frac{2}{\Vert \eta\Vert^2 }\text{Im}\Big(
\de \pi(X)\eta-\frac{\scalar{\de \pi(X)\eta}{\eta}}{\Vert \eta\Vert^2 }\eta, 
\de \pi(Y)\eta -\frac{ \scalar{\de \pi(Y)\eta}{\eta}} {\Vert\eta\Vert^2}\eta \Big)  \\
&=\frac{2}{\| \eta \|^2 } \im\scalar{\de \pi(X)\eta}{\de \pi(Y)\eta}\\
& = \frac{1}{i\Vert \eta\Vert^2} \scalar{\de \pi([X, Y]) \eta}{ \eta}\\
& = \dual{J_\pi([\eta])}{[X, Y]}, 
\end{align*}
where we have  used that $\scalar{\de \pi(X)\eta}{\eta}$ is purely imaginary. 
Therefore, $\gg^{\perp_{\theta_1}} =\gg(J_\pi([\eta]))$, and 
we have thus obtained that  $\omega^\Omega$ is symplectic if and only if 
$ \gg_{[\eta]}= \gg (J_\pi([\eta]))$. 
\end{proof}

\subsection{Coherent state representations} \label{sec:coherentstate}
This subsection is devoted to the question which exponential Lie groups admit coherent state representations and which representations are the coherent state representations.
Throughout this subsection,  let $G = \exp(\mathfrak{g})$ be an exponential Lie group and let $\pi$ be an irreducible representation of $G$. 
We recall that the representation $\pi$ is said to be a (symplectic) \emph{coherent state representation} if there exists $\eta \in \Hpi^{\infty} \setminus \{0\}$ such that $G \cdot [\eta]$ is symplectic.

We start with the following simple consequence of Lemma \ref{lem:symplectic}.

\begin{lemma}\label{symp-cor}
Let $G$ be an exponential Lie group and $\pi$ be an irreducible representation of $G$. For $\eta \in \Hpi^{\infty} \setminus \{0\}$, the orbit $G \cdot [\eta]$ is symplectic if and only if $\pker(\pi) = G(J_{\pi} ([\eta]))$.
\end{lemma}
\begin{proof}
By
Lemma~\ref{lem:symplectic} and since $G$ is exponential, 
thus $G(J_{\pi}([\eta]))$ is connected,  we get that
$G \cdot [\eta]$ is symplectic if and only if $G_{[\eta]}= G(J_{\pi} ([\eta]))$.
Now the result is a consequence of Proposition~\ref{prop:stabilizer_kernel}.
\end{proof}

The link between the moment map and the coadjoint orbit associated to $\pi$ is 
encoded in the fundamental identity
\begin{equation}\label{moment}
\overline{\Ran J_\pi}=\overline{\text{conv} \;\Oc_\pi};
 \end{equation}
 see \cite[Cor.~8, p.~274]{arnal1992convexity}. 

The next proposition uses the facts above to give the relation between the coherent state representations and the representations that are square-integrable representations modulo their projective kernel, for the case of general exponential Lie groups. 

\begin{proposition} \label{prop:non1mod}
 Let $G$ be an exponential Lie group and  let $\pi\colon G \to \mathcal{U}(\Hpi)$
 be an irreducible representation of $G$.
 \begin{enumerate}[\rm (i)]
     \item\label{non1mod_i}
     If  $\pi$  is square-integrable modulo $\pker(\pi)$, then it is a coherent state representation.
     \item \label{non1mod_ii}
      If $\pi$ is  a coherent state representation, then $G$ has a representation $\theta$ that is square-integrable modulo its projective kernel and $\pker(\theta)=\pker(\pi)$. 
      In addition, if the coadjoint orbit $\Oc_\pi$ associated to $\pi$ is of maximal dimension, then $\pi$ is square-integrable modulo $\pker(\pi)$.
 \end{enumerate}
\end{proposition}
\begin{proof}
Assume that $\pi$ is square-integrable modulo $\pker (\pi)$. 
Then its coadjoint orbit $\Oc_\pi$ is open in $\ell+\gg(\ell)^\perp$ by Proposition~\ref{prop:sq}, and hence
its affine hull is $\ell+\gg(\ell)^\perp$. 
The set $\overline{\Ran J_\pi}$ is closed in $\ell +\gg(\ell)^\perp$
and contains the open subset $\Oc_\pi$, by Equation \eqref{moment}. 
In particular, this shows that $\Ran J_\pi \cap \Oc_\pi \ne \emptyset$.
Thus, there exists $\eta\in \Hc^\infty \setminus \{0\}$ such that $\Ad^*(G) (J_\pi([\eta]))=\Oc_\pi$, and hence $G(J_\pi([\eta]))=\pker(\pi)$ by Corollary~\ref{sq-cor}. The claim follows now from Lemma~\ref{symp-cor}.

If $\pi$ is a coherent state representation of $G$, then there exists $\eta \in \Hpi^{\infty} \setminus \{0\}$ such that $G \cdot [\eta]$ is symplectic. By Lemma \ref{symp-cor}, this implies that $G(J_{\pi} ([\eta])) = \pker(\pi)$ is a normal subgroup. Therefore, the irreducible representation $\theta=\pi_{J_{\pi}([\eta])}$ associated to  $J_{\pi} ([\eta]) \in \mathfrak{g}^*$ is square-integrable modulo its projective kernel by Proposition \ref{prop:sq}.
Assume, in addition,  that $\Oc_\pi$ is of maximal dimension. 
On the one hand, for arbitrary $\ell\in \Oc_\pi$, we have $G(J_\pi([\eta]) =\pker(\pi)\subseteq G(\ell)$, and hence 
$\gg(J_\pi([\eta])) \subseteq \gg(\ell)$.
On the other hand, since $\Oc_\pi$ has maximal dimension, $\dim(\gg(J_\pi([\eta]))) \ge \dim (\gg(\ell))$.
It follows that $\gg(J_\pi([\eta]))=\gg(\ell)$, therefore $G(J_\pi([\eta])) =\pker(\pi)= G(\ell)$.
Hence, $\pi$ is square-integrable modulo $\pker(\pi)$.
\end{proof}

In general, not every coherent state representation is square-integrable modulo its projective kernel, as we will show in Example~\ref{ex:counterexample_nonunimodular}. 
This is, however, true for exponential groups $G$ and representations $\pi$ such that $G/\pker(\pi)$ is unimodular.
In fact, we have the following characterisation.

\begin{theorem} \label{thm:symplectic}
 Let $\pi$ be an irreducible representation of an exponential Lie group $G$ such that $G/\pker(\pi)$ is unimodular. Then the following assertions are equivalent:
 \begin{enumerate}[\rm (i)]
\item\label{lis_i} There exists $\eta \in \Hc^\infty\setminus \{0\}$ such that $G\cdot [\eta]$ is symplectic.
\item\label{lis_ii} For every $\eta \in \Hc^\infty\setminus \{0\}$ the orbit $G\cdot [\eta]$ is symplectic.
\item\label{lis_iii} $\pi$ is square-integrable modulo $\pker (\pi)$.
 \end{enumerate}
\end{theorem}

\begin{proof}
We can assume that $\pi=\pi_\ell= \ind_H^G\chi_\ell$, with 
$\ell\in \Oc_\pi= \Ad^*(G)\ell$. 

Let $\mathfrak{p}_{\pi}$ be the Lie algebra of $\pker(\pi)$. Then,
for every $X \in \pg_\pi$ and $\eta\in \Hc^\infty \setminus \{0\}$,  we have that
$\pi(\exp X)\eta =\ee^{\ie \dual{\ell}{X}}\eta$
(cf. the proof of \cite[Thm.~2.1]{bekka1990complemented}). Hence, 
$$ \de \pi(X)\eta =\frac{\de }{\de t} \pi(\exp tX)\eta\vert_{t=0}=
\ie \dual{\ell}{X} \eta,
$$ 
and
$$ \dual{J_\pi(\eta)}{X} =\dual{\ell}{X} \; \text{ for every } X\in \pg_\pi.$$
Since $\pg_\pi$ is an ideal, this implies that
\begin{equation}\label{1002}
\begin{gathered}
J_\pi(\eta)\vert_{\pg_\pi}=\ell\vert_{\pg_\pi} \quad \text{and} \quad
\Ad^*(G) J_\pi(\eta) \subseteq \ell +\pg_\pi^\perp,
\end{gathered}
\end{equation}
for every $\eta \in \Hpi^\infty\setminus \{0\}$.

\eqref{lis_i} $\Rightarrow$ \eqref{lis_iii}
Let $\eta \in \Hc^\infty\setminus \{0\}$ be such that the orbit $G\cdot [\eta]$ is symplectic. 
Then, by Lemma \ref{symp-cor}, it follows that
$G(J_\pi([\eta]))=\pker(\pi)$, and therefore $G(J_\pi([\eta]))$ is a normal subgroup of $G$. Thus, $\pi_{J_{\pi}([\eta])}$ is square-integrable modulo its projective kernel by Proposition \ref{prop:sq}. Moreover, an application of \cite[Lem.~5.3.7]{fujiwara2015harmonic} yields that $ \pker(\pi_{J_{\pi} ([\eta])}) = G(J_\pi ([\eta])) = \pker(\pi)$. 
Using Lemma~\ref{rose}  and Equation \eqref{1002}, it follows therefore that
$$ \Ad^* (G) J_\pi([\eta]) = J_\pi([\eta])  +\pg_\pi^\perp
=\ell  +\pg_\pi^\perp \subseteq \Ran J_\pi,
$$
since $\Ran J_\pi$ is $\Ad^*$ invariant.
On the other hand, for any $\eta' \in \Hc^\infty\setminus \{0\}$, 
$$  \Ad^* (G) J_\pi([\eta']) \subseteq J_\pi([\eta'])  +\pg_\pi^\perp
=\ell  +\pg_\pi^\perp,$$
by Equation \eqref{1002} again. 
Thus,
$$  \Ran J_\pi =\ell  +\pg_\pi^\perp.$$
This is already a closed subset of $\gg^*$ so, by Equation \eqref{moment}, it follows that
$$ \overline{\text{conv}\, \Oc_\pi}=\ell +\pg_\pi^\perp.$$
In particular, this implies that $\Oc_\pi \subseteq \ell+\pg_\pi^\perp= \Ad^*(G)J_\pi([\eta])$.
Since coadjoint orbits are either disjoint or equal,
it follows that  $\Oc_\pi= \Ad^*(G)J_\pi([\eta])$, and hence $\pi$ is square-integrable modulo $\pker(\pi)$ by Proposition \ref{prop:sq}.

\eqref{lis_ii} $\Rightarrow$ \eqref{lis_i} is trivial.

\eqref{lis_iii} $\Rightarrow$ \eqref{lis_ii}
Assume that $\pi$ is square-integrable modulo $\pker(\pi)$. 
Then, by Lemma~\ref{rose}, the coadjoint orbit $\Oc_\pi$ is equal to its affine hull, and hence an application of Corollary \ref{sq-cor} gives
$$G(\ell) =\pker(\pi)  \; \text{ for every }\; \ell\in \Oc_\pi=\overline{\text{conv}\Oc_\pi}= \overline{\Ran J_\pi}.
$$
Combining this with Proposition \ref{prop:stabilizer_kernel} gives $G(J_\pi(\eta))=\pker(\pi) = G_{[\eta]}$ for every $\eta\in \Hpi^\infty\setminus\{0\}$. Thus, each orbit $G \cdot [\eta]$, $\eta \in \Hpi^{\infty} \setminus \{0\}$, is symplectic by Lemma \ref{symp-cor}.
\end{proof}

Theorem~\ref{thm:intro} is now a consequence of Theorem~\ref{thm:symplectic}.

 \begin{proof}[Proof of Theorem~\ref{thm:intro}]
  Assume that the kernel of $\pi$ is discrete.  
     Then $\ker(\pi)$ is a central subgroup of $G$, hence 
     $G/\ker(\pi) $ is unimodular.
     On the other hand, by \cite[Thm.~2.1]{bekka1990complemented}, we have either that $\pker(\pi)=\ker(\pi)$ or that $\pker(\pi)/\ker(\pi)$ is nonempty and compact. 
     In the first case, $G/\pker(\pi)$ is clearly unimodular. 
     If $\pker(\pi)/\ker(\pi)$ is compact, then it is unimodular. 
     Since it is the centre of the group $G/\ker(\pi)$, it follows that $G/\pker(\pi)= (G/\ker(\pi))/(\pker(\pi)/\ker(\pi))$ is also unimodular, see, e.g., \cite[Rem.~6]{aniello2006square}.

     For proving the theorem, it therefore remains to show that under the hypotheses of the theorem, $\pi$ is square-integrable modulo 
     its projective centre if and only if it is square-integrable modulo $Z(G)$.
     Indeed, $\pi$ is square-integrable modulo $\pker(\pi)$ if and only if  $\pi'\colon G/\ker(\pi)\to \Uc(\Hc)$ is square-integrable (\cite[Cor.~2.1]{bekka1990complemented}). 
     On the other hand, we have that the representation $\pi'\colon G/\ker(\pi)\to \Uc(\Hc)$ is square-integrable if and only if $\pi''\colon G/(\ker (\pi))_0 \to \Uc(\Hc)$ is square-integrable modulo the centre of $G/(\ker(\pi))_0$, where 
     $(\ker (\pi))_0$ is the connected component of the identity of $\ker (\pi)$, see, e.g., \cite[Cor.~1.2.3]{moscovici1978coherent}. 
     Since $\ker (\pi) $ is discrete the last assertion is equivalent with the fact that $\pi$ is square-integrable modulo $Z(G)$. 

Assume that one (and then all) of  the equivalent  conditions in the statement holds. 
      If $Z(G)=\{0\}$, then $\pi$ is square-integrable in the strict sense. However, since $G$ is unimodular, this is impossible (see \cite[Cor.~4.2]{beltita2023squareintegrable}), therefore $Z(G)$ must be nontrivial.
    \end{proof}

As already mentioned above, the implication \eqref{lis_i} $\Rightarrow$ \eqref{lis_iii}  of   Theorem~\ref{thm:symplectic}
might fail for general exponential Lie groups. This is demonstrated by the following example.

\begin{example} \label{ex:counterexample_nonunimodular}
    Let $G$ be the exponential Lie group with Lie algebra 
$$\gg= \mathrm{span} \{A, B, P, R, Q, S\},$$
where 
$$ 
\begin{gathered}
\relax
[P,Q] = R,\,  [P,R] = S,
\\
  [A, P] = \frac{1}{2} P, \, [A, Q]=0, \,  [A,R] =\frac{1}{2} R, \, 
  [A, S]=S,\\
  [B,P] = -\frac{1}{2}P, \, [B, Q]=Q, \, [B,R] = \frac{1}{2}R, [B, S]=0.
  \end{gathered}
$$
Let $A^*$, $B^*$, $P^*$, $Q^*$, $R^*$, $S^*$ be the dual basis of $\{A, B, P, R, Q, S\}$ in $\mathfrak{g}^*$.

We first note that $G$ is not unimodular. Indeed, if $\Delta_G$ denotes the modular function of $G$, then 
$\Delta_G(\exp X)=\exp \tr(\ad_X)$ for every $X\in \gg$.
Since $\tr (\ad_A)= 2$,  the group $G$ is not unimodular.

We next show that there exists an irreducible unitary representation of $G$ admitting a symplectic coherent state orbit, but that fails to be square-integrable. For this, let $\ell=B^*+ S^*$ and $\Oc_\ell= \Ad^*(G)\ell$, and denote by $\pi=\pi_\ell\colon G\to \Uc(\Hc)$  a realisation of an irreducible representation corresponding to $\Oc_\ell$.
Then $\gg(\ell) =\RR B +\RR Q$, which is not an ideal of $\gg$, hence $\pi_\ell$ is not square-integrable modulo its projective kernel by Proposition \ref{prop:sq}.

Let $\pg$ be an ideal of $\gg$ contained in $\gg(\ell)$.
Assume towards a contradiction that $\pg \ne \{0\}$.
For $s, t\in \RR$ satisfying $s^2+t^2\ne 0$, it follows that $[sB+tQ, Q]=sQ$ and $[sB+tQ, B]=-tQ$, and hence $Q\in \pg$. 
Since $[P, Q]= R$, it follows that $R\in \pg\subseteq \RR B +\RR Q$. 
This is a contradiction, therefore $\pg=\{0\}$. Since the Lie algebra $\pg_{\pi}$ of $\pker(\pi)$ is an ideal contained in $\mathfrak{g}(\ell)$, it follows in particular that $\pker(\pi)$ is trivial.
Hence $G/\pker{\pi}$ is not unimodular.

We claim that  there is $\eta\in \Hpi^\infty\setminus \{0\}$
such that  $\Ad^*(G) J_\pi ([\eta])$ is open in $\gg^*$, so that $G(J_\pi([\eta]))$ is trivial, hence equal to $\pker(\pi)$.
If our claim is proved, then  the action of $\pi$ defines a symplectic orbit $G\cdot[\eta]$, by Lemma~\ref{symp-cor}.

We now prove our claim. 
The coadjoint orbit $\Oc_\ell$ is open, of dimension $4$, and
given by 
$$
\Oc_\ell = \{ sA^* +(1-\frac{pr}{2})B^* + \ee^{-a/2}rP^* +\frac{p^2}{2}Q^*
+ \ee^{-a/2}(-p)R^* + \ee^{-a}S^*\mid s, p, r, a\in \RR\};
$$
see \cite[(4d-3), p. 259]{inoue2015solvable}.
Then, for every $p, a \in \RR$ ($s=r=0$ above), 
$$
\begin{aligned}
\ell_1&  = B^*+\frac{p^2}{2}Q^*
+ \ee^{-a/2}(-p)R^* + \ee^{-a}S^* \in \Oc_\ell, \\
\ell_2&= B^*+\frac{p^2}{2}Q^*
+ \ee^{-a/2}p R^* + \ee^{-a}S^*\in \Oc_\ell.
\end{aligned}
$$
Hence, 
$$ \frac{1}{2}(\ell_1+\ell_2)= B^* +\frac{p^2}{2}Q^* + \ee^{-a}S^*\in \text{conv} (\Oc_\ell)$$ 
for every $p, a\in \RR$. 
Taking $p=\sqrt 2$, $a=0$, it follows that 
\begin{equation}
f= B^*+Q^*+S^* \in \text{conv} (\Oc_\ell).
\end{equation}

On the other hand, the coadjoint orbit $\widetilde{\Oc}= \Ad^*(G)(Q^*+S^*)$ 
is open, of dimension $6$ and given by
$$
\begin{aligned}
\widetilde{\Oc} = \{ s_1 A^* + r_1 B^* + q_1 P^* + \ee^{-b_1}(1+ \frac{p_1^2}{2})Q^* &
+ \ee^{-(a_1+b_1)/2}(-p_1)R^* + \ee^{-a_1}S^*
\\& \mid s_1, p_1, q_1, r_1, a_1, b_1\in \RR\};
\end{aligned}
$$
see \cite[(6d), p. 259]{inoue2015solvable}.
Taking $p_1=q_1=s_1=a_1=b_1=0$, $r_1=1$, we see that $f \in \widetilde{\Oc}$, and thus
\begin{equation}\label{1001}
f \in 
\widetilde{\Oc}\cap \conv(\Oc_\ell).
\end{equation}
Hence, since
$ I_{\pi}:=\overline{\Ran J_{\pi}} = \overline{\conv(\Oc_{\ell})}$ by the identity \eqref{moment}, it follows that $f\in I_\pi\cap \widetilde{\Oc}$.
The fact that $\widetilde{\Oc}$ is open implies that
$\Ran J_{\pi}\cap \widetilde{\Oc} \ne \emptyset$, so that
there exists $\eta\in \Hc^\infty\setminus 0$ such that $J_{\pi}([\eta]) \in \widetilde{\Oc}$.
Thus  $\widetilde{\Oc} = \Ad^*(G)J_{\pi}([\eta])$, and the coadjoint orbit $\Ad^*(G)J_{\pi}([\eta])$ is open, and this proves our claim.
\end{example}

\section{Application: Perelomov's completeness problem} \label{sec:perelomov}
This section describes an application of Proposition \ref{prop:stabilizer_kernel} to a problem considered in \cite{perelomov1972coherent} regarding the completeness of coherent state subsystems; see  \cite[p. 226]{perelomov1972coherent}. More precisely, we show that necessary conditions for the completeness of coherent state subsystems of exponential Lie groups can be obtained from criteria for the cyclicity of restrictions of associated projective representations established in \cite{romero2022density, enstad2022density, bekka2004square}.

\subsection{Overcomplete coherent states}
Let $\pi$ be an irreducible unitary representation of a Lie group $G$. 
For a  nonzero vector  $\eta \in \Hpi$, let $G_{[\eta]} $ be its  projective stabiliser group. 
Denote by $X = G / G_{[\eta]}$ the associated homogeneous $G$-space and let $s\colon  X \to G$ be a measurable cross-section for the projection $p\colon G \to X$. 
Assume that $X$ admits a $G$-invariant Radon measure $\mu_X$ and that $\eta$ is admissible, in the sense that
\[
 \int_X \vert \scalar{f}{\pi(s(x)) \eta } \vert^2 \; d\mu_X (x) < \infty.
\]
In this situation, following  \cite[Sect. 1.1]{moscovici1978coherent}, the vector $\eta$ is said to define a \emph{$\pi$-system of coherent states based on $X = G/G_{[\eta]}$}. 
Given such a vector $\eta$, there exists unique $d_{\pi, \eta} > 0$ such that
\begin{align} \label{eq:ortho}
 \int_{G/G_{[\eta]}} \vert \scalar{ f}{ \pi (s(x)) \eta }\vert ^2 \; d\mu_{X} (x) = d^{-1}_{\pi, \eta} \| f \|_{\Hpi}^2 \quad \text{for all} \quad f \in \Hpi;
\end{align}
see, e.g., \cite[Thm~1.2]{neeb1997square}. 

In many situations (i.e., when singletons in $X$ are $\mu_X$-null sets), the relation \eqref{eq:ortho} implies that the coherent state system $\{ \pi(s(x)) \eta \}_{x \in G/G_{[\eta]} }$ is overcomplete, in the sense that it remains complete in $\Hpi$ after the removal of an arbitrary element.

\subsection{Coherent state subsystems}
Let $\Gamma$ be a discrete subgroup of $G$ such that the factor space $X/\Gamma$ has finite measure. 
In \cite[p. 226]{perelomov1972coherent}, Perelomov considered the question of providing criteria for the completeness of a subsystem of coherent states
\begin{align} \label{eq:coherent_subsystem}
 \{ \pi (s(\gamma')) \eta : \gamma' \in \Gamma' \}
\end{align}
associated with $\Gamma' := p (\Gamma)$, in terms of the volume of $X/\Gamma$ and the admissibility constant $d_{\pi, \eta}$.

As a combination of Proposition \ref{prop:stabilizer_kernel} and results in \cite{romero2022density, bekka2004square}, the following theorem provides a necessary condition for the completeness of coherent state subsystems of the form \eqref{eq:coherent_subsystem} in the case of an exponential Lie group.

\begin{theorem} \label{thm:perelomov}
Let $G$ be an exponential Lie group and let $(\pi, \Hpi)$ be an irreducible representation of $G$ admitting an admissible vector. Suppose that $\Gamma$ is a discrete subgroup of $G$ such that $\Gamma' := p(\Gamma)$ is a uniform subgroup of $X = G/G_{[\eta]}$.
Let $s \colon X \to G$ be a Borel section. 

 If there exists $\eta \in \Hpi$ such that $\{ \pi (s(\gamma')) \eta : \gamma' \in \Gamma' \}$ is complete in $\Hpi$, then
\[
 \vol(X / \Gamma) d_{\pi, \eta} \leq 1,
\]
where $d_{\pi, \eta} = d_{\pi} > 0$ is the unique constant appearing in \eqref{eq:ortho} and is independent of $\eta \in \Hpi$.

The conclusion is independent of the choice of Borel section and choice of $G$-invariant Radon measure on $X$. 
\end{theorem}
\begin{proof}
By Proposition \ref{prop:stabilizer_kernel}, it follows that $G_{[\eta]} = \pker(\pi)$ for any $\eta \in \Hpi \setminus \{0\}$, so that $X = G/\pker(\pi)$ is an exponential Lie group. The existence of a uniform subgroup $\Gamma'$ in $X$ implies that $X$ must be unimodular, see, e.g., \cite[Prop. B.2.2]{bekka2008kazhdan}. In addition, the existence of an admissible vector $\eta$ means that $\pi$ is square-integrable modulo $\pker(\pi)$. As such, the map $\pi' := \pi \circ s$ defines an irreducible unitary projective representation of $X$ that is square-integrable in the strict sense. The constant $d_{\pi, \eta}$ in \eqref{eq:ortho} coincides with the formal degree $d_{\pi'}$ of $\pi'$, in the sense of \cite[Sect. 2.2]{romero2022density}. In particular, $d_{\pi, \eta} = d_{\pi'}$ is independent of $\eta \in \Hpi$.
An application of \cite[Thm. 7.4]{romero2022density} yields that $\vol(X/\Gamma) d_{\pi'} \leq 1$ whenever $\pi'(\Gamma') \eta = \{ \pi (s(\gamma')) \eta : \gamma' \in \Gamma' \}$ is complete. 

For the independence claims, note that if $\sigma : X \to G$ is another choice of Borel section, then the projective representations $\pi' := \pi \circ s$ and $\rho := \pi \circ \sigma$ are ray equivalent (cf. \cite[Prop. 3]{aniello2006square}), in the sense that there exists $\omega : X \to \mathbb{T}$ such that $\pi' (x) = \omega(x) \rho(x)$ for all $x \in X$, so that if $\pi'(\Gamma') \eta$ is complete, then so is $\rho(\Gamma') \eta$, and vice versa. 
Moreover, if $\mu_X$ is a Haar measure on $X$ and $\mu'_X = c \cdot \mu_X$ for some $c > 0$, 
then the volume $\vol'(X/\Gamma')$ and formal degree $d'_{\pi}$ relative to $\mu_X'$ are given by $\vol'(X/\Gamma') = c \cdot \vol(X/\Gamma')$ and $d'_{\pi} = 1/c \cdot d_{\pi}$, so that the product $\vol (X/\Gamma') d_{\pi} = \vol'(X/\Gamma') d'_{\pi}$.
\end{proof}

Theorem \ref{thm:perelomov} provides an extension of \cite[Thm. 1.2]{velthoven2022completeness} from nilpotent Lie groups to general unimodular exponential Lie groups. In addition, Theorem \ref{thm:perelomov} is valid for an arbitrary admissible vector $\eta \in \Hpi$, whereas \cite[Thm. 1.2]{velthoven2022completeness} required the orbit $G \cdot [\eta]$ to be symplectic, which in particular implies $\eta  \in \Hpi^{\infty}$.

Lastly, it is of interest to compare Theorem \ref{thm:perelomov} to density conditions for coherent state subsystems of semisimple Lie groups. In the latter setting, the general density conditions for restricted representations \cite{bekka2004square, enstad2022density, romero2022density} can be improved to depend on the projective stabiliser of an admissible vector, see, e.g., \cite[Theorem 4.5]{caspers2022density}. On the other hand, in the setting of exponential Lie groups,  
the projective stabilisers are always independent of the choice of vector, by Proposition \ref{prop:stabilizer_kernel}.

\section*{Acknowledgement} 
 The first named author has been supported by Research Grant GAR 2023 (cod 114), supported from the Donors' Recurrent Fund of the Romanian Academy, managed by the "PATRIMONIU" Foundation.
 
 For the second named author, this research was funded in whole or in part by the Austrian Science Fund (FWF): 10.55776/J4555.

The authors are grateful to D. Arnal for helpful discussions on various aspects of the paper and for providing Example \ref{ex:stabilizer}.
Also, the authors thank the referee for their comments that helped improving the manuscript.

\end{document}